\documentclass[12pt,twoside,reqno]{amsart}
\usepackage{amsfonts}
\usepackage{amsmath,amscd}
\usepackage{amsthm}
\usepackage{latexsym}
\usepackage{amssymb}
\usepackage{enumerate}
\newtheorem{theorem}{Theorem}
\newtheorem{Proposition}[theorem]{Proposition}
\newtheorem{Coro}[theorem]{Corollary}
\newtheorem{lemma}[theorem]{Lemma}

\newcommand{\IO}{\int_{\partial C_+^p}}
\newcommand{\ie}{{\it{i.$\,$e. }}}
\newcommand{\R}{\mathbb{R}}
\newcommand{\C}{\mathbb{C}}
\newcommand{\N}{\mathbb{N}}
\newcommand{\e}{\mathrm{e}}
\newcommand{\re}{\mathrm{Re}}
\newcommand{\im}{\mathrm{Im}}

\begin{document}
\title{Special values of Dirichlet series and zeta integrals}

\author{Eduardo Friedman}
\address{Departamento de Matem\'atica \\
 Universidad de Chile \\ Casilla 653, Santiago 1, Chile}
\email{friedman@uchile.cl}

\author{Aldo Pereira}
\address{Departamento de Matem\'atica \\
 Universidad de Chile \\ Casilla 653, Santiago 1, Chile}
\email{pereirasolis@gmail.com}

\thanks{This work was partially supported by  Chilean Fondecyt
grants 1085153 and 1040585 and
 Programa Iniciativa  Cient\'{\i}fica Milenio   grant ICM P07-027-F}

\begin{abstract} For  $f$ and $g$ polynomials in $p$ variables,  we relate the
special  value at a non-positive integer $s=-N$,  obtained by
analytic continuation  of the Dirichlet series
$$
\zeta(s;f,g)=\sum_{k_1=0}^\infty \cdots \sum_{k_p=0}^\infty
g(k_1,\dots,k_p)f(k_1,\dots,k_p)^{-s}\ \,(\re(s)\gg0),
$$ to special
values of zeta integrals $$ Z(s;f,g)=\int_{x\in[0,\infty)^p}
g(x)f(x)^{-s}\,dx  \, \ (\re(s)\gg0).$$ We prove a simple relation
between  $\zeta(-N;f,g)$ and $Z(-N;f_a,g_a)$, where for $a\in\C ^p,\
f_a(x)$ is the shifted polynomial $f_a(x)=f(a+x)$.
 By direct calculation we prove the product rule for zeta integrals at $s=0$,
$
\mathrm{degree}(fh)\cdot Z(0;fh,g)=\mathrm{degree}(f)\cdot
Z(0;f,g)+\mathrm{degree}(h)\cdot Z(0;h,g),
$
and   deduce the corresponding rule for Dirichlet series at $s=0$,
$
\mathrm{degree}(fh)\cdot\zeta(0;fh,g)=\mathrm{degree}(f)
\cdot\zeta(0;f,g)+\mathrm{degree}(h)\cdot\zeta(0;h,g).
$
 This last formula generalizes work of Shintani and Chen-Eie.
\end{abstract}

\maketitle

\section{Introduction}

\noindent We shall deduce special values of
 Dirichlet series
\begin{equation}\label{DefZeta}
\zeta(s;f,g):=\sum_{k_1,\dots,k_p=0}^\infty
g(k_1,\dots,k_p)f(k_1,\dots,k_p)^{-s}\quad(\re(s)\gg0)
\end{equation}
from those of  zeta integrals
\begin{equation}\label{DefIntegrals}
 Z(s;f,g):=\int_{x_1=0}^\infty\cdots\int_{x_p=0}^\infty
g(x_1,\dots,x_p)f(x_1,\dots,x_p)^{-s}\,dx_p\cdots\, dx_1\, .
\end{equation}
Here $f$  and $g$ are polynomials  in $p$ variables with complex
coefficients, with some restrictions on  $f$   to ensure
 the existence of an appropriate branch of $\log f$ and the
  convergence and analytic continuation of sums and integrals.

The use of integrals to express sums goes back to Euler's invention
of the Euler-MacLaurin formula to compute
$\zeta(2)=\sum_{n=1}^\infty n^{-2}$ numerically \cite{16}. Later
authors, such as Mellin \cite{12}, Mahler \cite{11}, Shintani
\cite{15}, Cassou-Nogu\`es \cite{1}, Sargos \cite{14}, Lichtin
\cite{10}, Essouabri \cite{7}, Peter \cite{13} and de Crisenoy
\cite{5}, have used various integrals to ascertain the existence of
a meromorphic continuation of $\zeta(s;f,g)$ and to compute its
residues and various expansions. As the Euler-MacLaurin formula
already shows, at a general $s$ the connection between Dirichlet
series $\zeta(s;f,g)$  and zeta integrals $Z(s;f,g)$ is rather
complicated. We will show, however, that at non-positive integers
$s=0,-1,-2,\ldots$ the relationship    becomes quite simple.

 Consider, as a first easy case, the Riemann zeta function
$$
\zeta(s)=\zeta(s;x+1,1)=\sum_{k=0}^\infty
\frac{1}{(k+1)^s}\qquad\qquad(\re(s)>1)
$$
and its corresponding zeta integral (convergent for $\re(s)>1,\
\,\re(a)>-1$)
$$
Z(s;x+1+a,1)=\int_{0}^\infty
\frac{dx}{(x+1+a)^s}=\frac{(1+a)^{1-s}}{s-1}\,.
$$
Here we have allowed ourselves to replace the polynomial $f(x)=x+1$
by its shift $f_a(x)=f(x+a)$. Using the above meromorphic
continuation in $s$ for $Z(s;x+1+a,1)$  we find
$$
Z(-N;x+1+a,1)=
\frac{(1+a)^{N+1}}{-N-1}=-\frac1{N+1}\sum_{\ell=0}^{N+1}\binom{N+1}{\ell}
a^\ell,
$$
for $N\ge0$ a non-negative integer. If we mindlessly replace every
occurrence of $a^\ell$ above by the Bernoulli number $B_\ell$ we
obtain
$$
\frac{-1}{N+1}\sum_{\ell=0}^{N+1}\binom{N+1}{\ell}
B_\ell=-\frac{B_{N+1}(1)}{N+1}= (-1)^N \frac{B_{N+1}}{N+1}=\zeta(-N)
$$
\cite[p.\ 76]{4} \cite[pp.\ 67--68]{9}, where  the Bernoulli
polynomials $B_j(t)$ are defined by
\begin{equation}\label{BernoulliDef}
B_0(t)=1,\qquad \frac{dB_j}{dt}= jB_{j-1}(t),\qquad \int_0^1
B_j(t)\,dt=0 \qquad (j\ge1),
\end{equation}
and the Bernoulli numbers as $B_j=B_j(0)$.
 In short, to compute the
value of the Dirichlet series $\zeta(s)=\zeta(s;x+1,1)$ at $s=-N$,
we simply take the polynomial (in $a$) giving the shifted zeta
integral $Z(-N;x+1+a,1)$ and replace powers of $a$ by Bernoulli
numbers.

This simple relation between Dirichlet series and shifted zeta
integrals holds quite generally, as we shall now describe. For
$h\in\C[x]=\C[x_1,\dots,x_p]$ and
 $a=(a_1,\dots,a_p)\in\R^p$, let
 $h_a(x):=h(x+a)$.\footnote{\ We use this notation only when the subscript is the letter $a$.
  For example, $f_j$ in Theorem \ref{MainTheorem} below simply stands for one of $n$ polynomials.}
   In \S2 we  prove (under some hypothesis on $f$) that
the maps  $a\to Z(-N;f_a,g_a)$ and $a\to \zeta(-N;f_a,g_a)$ are
 polynomials in $a$. Here $a$ ranges in a small
enough ball in $\R^p$ containing the origin, while $N$, $f$  and $g$
are fixed.  To express the relation between the two polynomials (in
$a$) $Z(-N;f_a,g_a)$ and $\zeta(-N;f_a,g_a)$, write the former as a
(finite!) sum of monomials
\begin{equation}\nonumber
Z(-N;f_a,g_a)=\sum_{L} c_L a^L \qquad\qquad\qquad\qquad
\Big(a^L:=\prod_{i=1}^p a_i^{L_i},\ \, c_L\in\C\Big).
\end{equation}
In \S2, Proposition \ref{RaabeMainP},  we prove that
\begin{equation}\label{InverseRaabe2}
\qquad \ \ \ \zeta(-N;f_a,g_a)=\sum_{L} c_L B_L(a) \ \ \ \qquad
\Big( B_L(a):=\prod_{i=1}^p B_{L_i}(a_i)\Big).
\end{equation}
 Taking
$a=0$ we obtain the special value of the Dirichlet series
\begin{equation}\label{EulerMacLaurinAtN}
\zeta(-N;f,g)=\sum_{L} c_L B_L \qquad\qquad\qquad\qquad
\Big(B_L:=\prod_{i=1}^p B_{L_i}\Big)
\end{equation}
in terms of   special values of zeta integrals and products of
Bernoulli numbers.  Equation \eqref{InverseRaabe2} explains the
profusion   of Bernoulli polynomials  in  Shintani's formulas
\cite{15}.

 Equation \eqref{InverseRaabe2} follows rather formally from the
 ``Raabe formula"\footnote{\ Raabe's 1843
formula is
$\int_0^1\log\!\big(\Gamma(x+t)/\sqrt{2\pi}\big)\,dt=x\log x.$ See
\cite[p.\ 367]{8} for the connection to \eqref{Raabe}. A $p$-adic
version of   Raabe's formula   was given in \cite{3}.} (see
Proposition \ref{RaabeMainP})
\begin{equation}\label{Raabe}
Z(s;f_a,g_a)=\int_{t\in[0,1]^p} \zeta(s;f_{a+t},g_{a+t})\,dt.
\end{equation}
This formula, though easily proved   by an ``unfolding" argument,
provides a powerful  link between zeta integrals and Dirichlet
series. The Raabe formula \eqref{Raabe} holds everywhere in $s\
\big($save at the poles of $Z(s;f,g)\big)$, but it can be inverted
at special values $s=-N$ to yield $\zeta(-N;f_a,g_a)$ in terms of
$Z(-N;f_a,g_a)$. Quite generally (see Lemma \ref{LemmaB}), two
polynomials $Q(a)$ and $P(a)=\sum_{L} d_L a^L$ in $p$ variables
  are linked by a Raabe formula
$$
\ \ \ \  \ \ \ P(a)=\int_{t\in[0,1]^p} Q(a+t)\,dt
$$
if and only if
$$
Q(a)=\sum_{L} d_L B_L(a).
$$

Our main motivation for relating special values of zeta integrals
and Dirichlet series is  that  integrals are usually easier to
handle. In \S3 we   give (under some hypothesis on $f$)  a slightly
complicated formula for  $Z(s;f,g)$, for $s$ a non-positive integer.
For $s=0$ we are able to simplify it enough to prove a formula for
the special value $Z\big(0;\prod_{j=1}^n f_j,g\big)$ in terms of the
individual $Z(0;f_j,g)$. In view of the relation between zeta
integrals and Dirichlet series at special values, we deduce an
analogous formula giving $ \zeta\big(0;\prod_{j=1}^n f_j,g\big)$ in
terms of the individual $\zeta(0;f_j,g)$.

 A   first case  of this formula was
proved by Shintani. Namely, if all the $f_j(x)$ are polynomials of
degree one, positive for all $x\in[0,\infty)^p$, Shintani \cite[p.\
206]{15} \cite[p.\ 386]{8} showed
\begin{equation}\label{Shintani}
\zeta\big(0;\prod_{j=1}^n
f_j,1\big)=\frac1n\sum_{j=1}^n\zeta(0;f_j,1).
\end{equation}
Shintani's formula \eqref{Shintani} cannot be expected to generalize
literally to higher-degree polynomials.\footnote{\ To see this,
assume \eqref{Shintani} and consider two ways to associate
$\prod_{j=1}^3 f_j$,
$$
    \zeta\big(0;f_1 (f_2 f_3),g\big) = \frac{\zeta(0;f_1, g)}{2}
      + \frac{\zeta(0;f_2 f_3,g)}{2}
    = \frac{\zeta(0;f_1, g)}{2}    + \frac{\zeta(0;f_2,g)}{4}
      + \frac{\zeta(0;f_3, g)}{4} ,
$$
  $$
    \zeta\big(0;(f_1 f_2) f_3,  g \big) = \frac{\zeta(0;f_1 f_2, g)}{2}
      + \frac{\zeta(0;f_3, g)}{2}
   = \frac{ \zeta(0;f_1,g)}{4}  + \frac{\zeta(0;f_2, g)}{4}
      + \frac{\zeta(0;f_3,g)}{2} .
 $$
 On subtracting, we find $\zeta(0;f_1, g) = \zeta(0;f_3, g)$.
  This would imply that $\zeta(0;f, g)$ does not depend on
  $f$, contradicting a host of known facts, e.g. \cite[Lemma 2]{15}.}
 Besides correcting  for the degree of the $f_j$, we need some kind of
  irreducibility condition on the $f_j$  not allowing them to factor
  into a product of polynomials in separate variables. Indeed,
  if we had
\begin{equation} \label{Decompose}
f(x_1,\dots,x_p) = h(x_1, \ldots, x_\ell) \cdot
\tilde{h}(x_{\ell+1}, \ldots, x_p)
\end{equation}
 for some $1\le\ell<p$, then from  \eqref{DefZeta}  we would find a
 corresponding factorization
  \begin{eqnarray} \label{pumpum}
    \zeta(s; f, 1) =\zeta(s; h, 1)\cdot \zeta(s;\tilde{h},1).
  \end{eqnarray}
This kind of relation (when applied by analytic continuation at
$s=0$) is inconsistent with simple generalizations of Shintani's
formula \eqref{Shintani}.\footnote{\ More precisely, it is
inconsistent with generalizations of the form
\begin{equation*}
   \zeta(0; f_1f_2\cdots f_n, 1) =  c_1\zeta(0; f_1, 1)+c_2\zeta(0; f_2, 1)+\cdots+c_n\zeta(0; f_n, 1),
  \end{equation*}
where the $c_\ell$ depend at most on the degrees of the $f_j$'s
$\,(1\le j\le n)$.}

Mahler \cite[p.\ 385]{11} gave a simple hypothesis on the polynomial
$f$ ensuring that it does not  separate  as in \eqref{Decompose}.
\vskip.25cm \noindent{\bf{Mahler's Hypothesis.}}   The polynomial
$f(x) \in \C[x_1, \ldots,
    x_p]$ is non-constant and does not
    vanish anywhere in the closed ``octant" $[0,\infty)^p$.
     Moreover, its top-degree homogeneous
    part $f_\mathrm{top}(x)\not=0$ for  $x\in[0,\infty)^p$,
    save at $x=(0,0,\dots,0)$.
\vskip.25cm \noindent Under his hypothesis Mahler \cite{11} showed
that $Z(s; f, g) $ and $\zeta(s;f,g)$ converge for $\re(s)\gg0$,
extend meromorphically in $s$ to all of $\C$ and are regular
 at the
non-positive integers $s=0,-1,-2,\dots\ $.\footnote{\  The branch of
$\log f$ used by Mahler in defining $Z(s;f,g)$ and $\zeta(s;f,g)$ is
just any one that is continuous on $[0,\infty)^p$. It exists
precisely because of the non-vanishing of $f$  that he assumed. The
value of $Z(s;f,g)$ and $\zeta(s;f,g)$ at integer values of $s$
proves to be independent of the choice of (continuous) branch of
$\log f$.} An advantage of Mahler's Hypothesis for our purposes is
that if $f$ and $h$ satisfy it, then
 so does $fh$.

\begin{theorem}\label{MainTheorem}
Let $g$ and $f_j\ (1\le j\le n)$
 be  polynomials in $p$ variables and assume that
all the     $f_j$  verify Mahler's Hypothesis above. Then the values
at $s=0$ obtained by analytic continuation of the Dirichlet series
$$
\zeta(s;f_j,g):=\sum_{k_1=0}^\infty \cdots \sum_{k_p=0}^\infty
g(k_1,\dots,k_p)f_j(k_1,\dots,k_p)^{-s}\qquad(\re(s)\gg0)
$$
satisfy the product rule
 \begin{equation}\label{SeriesFormula}
 \mathrm{degree}\bigg(\! \prod_{j=1}^n
f_j\bigg)\cdot\zeta\big(0;\prod_{j=1}^n f_j,g\big)=\sum_{j=1}^n
\mathrm{degree}\hskip.05cm(f_j)\cdot\zeta(0;f_j,g).
\end{equation}
 \end{theorem}

\begin{Coro} $\mathrm{(Chen}\text{-}\mathrm{Eie)}$
 Assume furthermore that
the  polynomials    $f_j$  all have the same degree. Then
\begin{equation*}
 \zeta\big(0;\prod_{j=1}^n f_j,g\big)=\frac1n\sum_{j=1}^n
 \zeta(0;f_j,g).\nonumber
\end{equation*}
\end{Coro}
\noindent Chen and Eie \cite[p.\ 3219]{2} do not explicitly mention
Mahler's Hypothesis, nor the condition of equal degree for the
$f_j$, but we have seen above that  assumptions of this kind are
unavoidable.\footnote{\  Chen and Eie  take their sums in
\eqref{DefZeta} for $k_i\ge1$ instead of $k_i\ge0$, but this is just
a matter of replacing $f(x)=f(x_1,\ldots,x_p)$  by
$f(x_1+1,\ldots,x_p+1)$,
  and similarly with $g$.}

\section{Dirichlet series and zeta integrals at special values}
 Recall from \S1 that a
polynomial $f\in\C[x]$ in $p$ variables
  satisfies Mahler's Hypothesis if
$m:=\deg(f)>0$, $\,f(x)\not=0$ for all
$x\in\R^p_{\ge0}:=[0,\infty)^p$, and if its top-degree homogeneous
part $f_\mathrm{top}$  satisfies $f_\mathrm{top}(x)\not=0$ for all
$x\in\R^p_{\ge0}-\big\{0\big\}$. We let
$\mathcal{M}=\mathcal{M}_{m,p}$ denote the set of all such $f$.
Certainly $\mathcal{M}$ is non-empty, as it contains the polynomial
$x_1^m+x_2^m+\cdots+x_p^m+1$.

 Together with $f\in\mathcal{M}$, we
will need to consider the shifted polynomial $f_a$ defined as
$f_a(x):=f(x+a)$, where $a\in\C^p$. Since
$\big(f_a\big)_\mathrm{top}=f_\mathrm{top}$, it is clear that for
$a\in \R^p_{\ge0}$, $f_a\in\mathcal{M}$ if $f\in\mathcal{M}$. We
will now show that $f_a\in\mathcal{M}$  for all $a$ in a small
enough neighborhood of the origin in $\C^p$. For this it suffices to
show that $\mathcal{M}$ is open in the finite-dimensional complex
vector space of all polynomials of degree $m$ in $p$ variables
(space of coefficients).

To show that $\mathcal{M}$ is open we first estimate $|f(x)|$ for
$x\in\R^p_{\ge0}$. It proves convenient to switch away from
cartesian coordinates $x=(x_1,x_2,\dots,x_p)$. Instead of the
well-known spherical co-ordinates  used by Mahler \cite{11} for this
purpose, we will use ``cubical" co-ordinates $(\rho,\sigma)$,
\begin{equation}\label{Cubical}
\rho=\rho(x):=\max(|x_1|,|x_2|,\dots,|x_p|),\quad
\sigma=\sigma(x):=\frac{x}{\rho(x)}\quad(x\not=0).
\end{equation}
We denote by $\partial C_+^p$ the piece of boundary of the
unit-hypercube $C^p=[0,1]^p$ where at least one co-ordinate is 1,
\begin{equation}\label{Hypercube}
\partial C_+^p:=\big\{x\in\R^p_{\ge0} \big|\  \rho(x)=1   \big\}.
\end{equation}
 For $f\in\mathcal{M}$
and $x\in\R^p_{\ge0},\ x\not=0,$  write
\begin{align}\nonumber
r(x)&=r_f(x):=\frac{f(x)-f_\mathrm{top}(x)}{f_\mathrm{top}(x)} , \\
\label{Factorf} f(x) &= f_\mathrm{top}(x)\big(1+r(x)\big)=\rho^m
f_\mathrm{top}(\sigma) \big(1+r(\rho\sigma)\big)
\qquad\qquad(x\not=0)  .
\end{align}
Note that
\begin{equation}\label{Writer}
 r(\rho\sigma)=r_f(\rho\sigma)=\frac{f_{(m-1)}(\sigma)}{\rho f_\mathrm{top}(\sigma)}+
\frac{f_{(m-2)}(\sigma)}{\rho^2
f_\mathrm{top}(\sigma)}+\cdots+\frac{f_{(1)}(\sigma)}{\rho^{m-1}
f_\mathrm{top}(\sigma)}+ \frac{f_{(0)}}{\rho^m
f_\mathrm{top}(\sigma)},
\end{equation}
where $f_{(j)}$ denotes the homogeneous part of $f$ of degree $j$.
Hence for all $\sigma\in \partial C_+^p$ and all $\rho>\rho_f$ (for
some large enough $\rho_f$) we have $|r_f(\rho\sigma)|<\frac12$
(say). Similarly, for all polynomials $\tilde f$ in a neighborhood
of  $f$ we   have $1+r_{\tilde f}(\rho\sigma)\not=0$ for
$\rho>\rho_f$. Considering the factorization \eqref{Factorf}, we see
that non-vanishing conditions on compact sets  insure Mahler's
Hypothesis for $\tilde f$. Thus $\mathcal{M}$ is
  open.\footnote{\ Here is a proof that the complement of such a set is closed. For $K\subset\C^p$ compact  let
  $$
W=W_{K,m}:=\Big\{f\in\C[x_1,...,x_p]\big|\ \deg(f)\le m,\ \exists
k\in K,\ f(k) =0\Big\}.
  $$
  It suffices to show that if $\{f_n\} \in W$ converges to $f$ (say, in the uniform norm),
  then $f\in W$. Let $f_n(k_n)=0$, with $k_n\in K$. Since $K$ is
  compact, there exists a subsequence $k_{n_j}$ converging to $k\in
  K$. But then $0=\lim_j f_{n_j}(k_{n_j})=f(k).$ Thus $f\in W$.}

Since $\R^p_{\ge0}$ is simply  connected and $f\in \mathcal{M}$ does
not vanish there, we can choose a continuous branch   $\log
f:\R^p_{\ge0}\to\C$. By the same token, locally around a given $f$
we can choose this branch so that it depends analytically on  the
coefficients of  $f$. Any other continuous choice of $\log f$ will
differ by    $2\pi i \ell$ for some fixed integer $\ell$,
introducing a factor of $\e^{-2\pi i \ell s}$   in our zeta
integrals \eqref{DefIntegrals} and series \eqref{DefZeta}. Hence
their values (or residues) at any integer $s$ are independent of the
branch chosen.

 Following Mahler \cite{11} and considering \eqref{Factorf},  we choose our branch of $\log f$ so that
 for $x\not=0,\ x\in\R_{\ge0}^p$ and $ f\in
\mathcal{M} $,
\begin{equation}\label{Branch}
\log f(x)= \log f(\rho\sigma)= m\log \rho+\log
f_\mathrm{top}(\sigma)+ \log\!\big(1+r(\rho\sigma)\big) ,
\end{equation}
where $\log \rho$ is real-valued, $\log f_\mathrm{top}(\sigma)$ is
any continuous choice of $\log f_\mathrm{top}$ on the (simply
connected) hypersurface $\partial C_+^p$, and $\log(1+r)$ is the
unique continuous branch which for large enough $\rho$ is given by
the principal value
$$
\log(1+r )=-\sum_{\lambda=1}^\infty
\frac{(-r)^\lambda}{\lambda}\qquad\qquad\qquad(r=r(\rho\sigma),\
\rho\gg0).
$$  In \eqref{Branch} we  used Mahler's Hypothesis  to insure
$ 1+r \not=0$ and $f_\mathrm{top}(\sigma)\not=0$.

We  now state Mahler's main result \cite{11} concerning the
meromorphic continuation  of $\zeta(s;f,g)$  and of  $Z(s;f,g)$.
\begin{theorem} \label{Mahler1}
$\mathrm{(Mahler)}$ Suppose $f$ is a polynomial satisfying
 Mahler's Hypothesis  and let $g$ be any
polynomial in the same number $p$ of variables. Then
\begin{equation}\label{Dirzeta}
Z(s;f,g):=  \int_{\R^p_{\ge0}} g(x)f(x)^{-s}\,dx
\end{equation}
and
\begin{equation}\label{Dirzeta2}
\zeta(s;f,g):=  \sum_{k_1=0}^\infty \cdots \sum_{k_p=0}^\infty
g(k_1,\dots,k_p)f(k_1,\dots,k_p)^{-s}
\end{equation}
 both converge absolutely and uniformly on compact subsets of the right half-plane $\re(s)>
(\deg(g)+p)/\deg(f)$, and extend to all of  $\,\C$ as   meromorphic
functions of $s$, regular at all non-positive integers
$s=0,\,-1,\,-2,\dots\,$. Their poles are all simple and occur among
the rational numbers of the form
$$\displaystyle{s=\frac{\deg(g)+p-\ell}{\deg(f)}},$$ with $\ell\ge0$
an integer. Moreover, $Z(s;f,g )$ and $\zeta(s;f,g )$ are analytic
in $s,$ $f$ and $g$, as long as $s$ stays outside the above  set of
possible poles and $f$ stays in an open simply connected subset of
$\mathcal{M}$.
\end{theorem}
\noindent Mahler does not explicitly address the analytic dependence
on (the coefficients of) $f$ and $g$, but it is immediate from his
proof. In \S3 the reader will find a full proof of Mahler's theorem
for $Z(s;f,g)$. As Mahler showed, the analytic continuation for
$\zeta(s;f,g)$ follows readily from that of $Z(s;f,g)$ and the
Euler-MacLaurin formula.

 We will need to go into Mahler's
proof to simplify his formulas for $Z(s;f,g)$ at special values.
Mahler actually dealt with the slightly different sums
$$
\sum_{k_1=1}^\infty \cdots \sum_{k_p=1}^\infty
g(k_1,\dots,k_p)f(k_1,\dots,k_p)^{-s}\qquad\qquad (\re(s)\gg0),
$$
\ie Mahler summed over $k_i\ge1$, whereas we use $k_i\ge0$ in
\eqref{Dirzeta2}. Our series can be written as a finite sum of
Mahler's, and so his theorem gives the meromorphic continuation of
\eqref{Dirzeta2}. The only point worth noting here is that
polynomials in fewer variables, obtained from $f$ (which we assume
 satisfies Mahler's Hypothesis) by inserting 0 for some $x_i$'s
 again satisfy Mahler's Hypothesis (in fewer variables) and have the same
degree as $f$.

We can now prove
\begin{Proposition}\label{RaabeMainP}
  Suppose  $ f$ and $g$ are polynomials in $p$ variables, and assume
$f$ satisfies Mahler's Hypothesis. Then,
\begin{enumerate}
\vskip.5cm
\item (Raabe formula) For $s$ outside the  possible pole set of $Z(s;f,g)$
 given in Mahler's Theorem above, we have
\begin{equation}\label{RaabeProp}
 \qquad\qquad Z(s;f,g)=\int_{t\in[0,1]^p}\zeta(s;f_t,g_t)\,dt,
\end{equation}
 where $f_t(x):=f(t+x)$ and $dt$ is Lebesgue measure on   $\R^p$.
 \vskip.3cm
\item For a fixed  integer $N\ge0$,
 the  maps  $a\to \zeta(-N;f_a,g_a)$ and $a\to  Z(-N;f_a,g_a)$
are    polynomials  in  $a=(a_1,\ldots,a_p) \in\R_{\ge0}^p$  of
degree at most $N\deg(f)+\deg(g)+p$.
 \vskip.3cm
 \item   If we write out the polynomial $Z(-N;f_a,g_a)$ as a sum of monomials,
\begin{equation}\nonumber
\qquad  Z(-N;f_a,g_a)=\sum_{L} c_L a^L \qquad
\Big(a^L:=\prod_{i=1}^p a_i^{L_i},\ \ c_L=c_L(N;f,g)\in\C\Big),
\end{equation}
 then
\begin{equation}\label{Bern}
 \zeta(-N;f,g)=\sum_{L} c_L B_L, \qquad\qquad\qquad\qquad\qquad\qquad\qquad\quad\
\end{equation}
where  $B_L:=\prod_{i=1}^p B_{L_i}$ is a product of  Bernoulli
numbers. More generally, for $a\in \R_{\ge0}^p$ we have
$$
\qquad\qquad \zeta(-N;f_a,g_a)=\sum_{L} c_L B_L(a),
$$
where $B_L(a)=\prod_{i=1}^p  B_{L_i}(a_i)$ is a product of Bernoulli
polynomials.
\end{enumerate}
\end{Proposition}
\noindent We note that the Raabe formula \eqref{Raabe}
 stated in
\S1  follows from   \eqref{RaabeProp} on replacing $f$ by $f_a$,
noting that $\big(f_a\big)_t=f_{a+t}$.

\begin{proof} Write $m:=\deg(f),\ q:=\deg(g)$. For $\re(s)>\frac{p+q}{m}$,
 the integral and series defining  $Z(s;f,g)$ and $\zeta(s;f,g)$
are absolutely convergent, as is clear from \eqref{Writer} and
\eqref{Factorf}. Note also that if $t\in\R_{\ge0}^p$, then $f_t $
satisfies Mahler's Hypothesis and $f_t$ has the same degree as $f$.
Let $\N_0:=\{0,1,2,3,\dots\}$ and compute for
$\re(s)>\frac{p+q}{m}$,
\begin{align}
\int_{[0,1]^p}&  \zeta(s;f_t,g_t)\,dt= \int_{[0,1]^p}
\sum_{k\in\N_0^p} g(k+t)f(k+t)^{-s}\,dt\\ =&
 \sum_{k\in\N_0^p}\int_{k+[0,1]^p}
g(t)f(t)^{-s}\,dt\nonumber= \int_{\R^p_{\ge0}}
g(t)f(t)^{-s}\,dt\nonumber =Z(s;f,g).
\end{align}
By analytic continuation, the Raabe formula $$\int_{[0,1]^p}
\zeta(s;f_t,g_t)\,dt= Z(s;f,g) $$ holds for all $s$ outside the
possible pole set  given in Mahler's Theorem.

To prove the polynomial nature of $a\to \zeta(-N;f_a,g_a)$ we follow
the proof sketched in \cite{6}.  Mahler's Theorem implies that
$\zeta(-N;f_a,g_a)$ is an analytic function of $a$ as  $a$ ranges in
some small open ball in $\C^p$ containing the origin. We   shall
show that all sufficiently high derivatives with respect to $a$
vanish.
\begin{lemma}\label{LemmaA}
    Let $ \partial_L$ be the differential operator
      $$ \qquad\qquad\partial_L :=\frac{\partial ^{|L|}}{\partial x_1^{L_1}\cdots \partial
x_p^{L_p}} \qquad \quad\big(L=(L_1,\dots,L_p),\ \ \
|L|:=\sum_{i=1}^p L_i \,\big),
$$
and let $f$ and $g$ be polynomials in $p$ variables of degree $m$
and $q$, respectively.
    Then, for $x=(x_1,\dots,x_p)$ in an open set in $\C^p$ where some branch of $\log f(x)$ is analytic,  we have
    \begin{equation} \label{derivada}
      \partial_L \big(  g(x)f(x)^{-s}\big)
      =
      \sum_{\nu=0}^{|L|}\Big(\prod_{j=0}^{\nu-1}(s+j)\Big)P_{L,\nu}(x)\,f(x)^{-(s+\nu)},
    \end{equation}
    where  $P_{L,\nu}(x)$ is a polynomial in $x$, independent of
    $s$,      of degree at most $q+m\nu -|L|$,   vanishing if $q+m\nu -|L|<0$.
\end{lemma}
\begin{proof}
      This is a straight-forward induction on $|L|:=\sum_{i=1}^p L_i$.
      We omit the routine details.\end{proof}

For $a\in\R_{\ge0}^p$, Mahler's theorem implies that
$$
\zeta\big(s;f_a, (P_{L,\nu})_a\,\big)=\sum_{k\in\N_0^p}
P_{L,\nu}(k+a)\,f(k+a)^{-(s+\nu)}
$$
converges (absolutely and uniformly on compact subsets) in the right
half-plane $\re(s+\nu)>\frac{p+\deg(P_{L,\nu})}{\deg(f)}$. Taking
$|L|\ge Nm+q+p+1$ and using the bound on $\deg(P_{L,\nu})$ in the
Lemma, this is the half-plane $\re(s)>-N-\frac1m$. Hence, for such
$L$ and $a$,
$$
\frac{\partial ^{|L|}\Big( \zeta(s;f_a,g_a) \Big)}{\partial
a_1^{L_1}\cdots \partial a_p^{L_p}}
 =
 \sum_{\nu=0}^{|L|}\Big(\prod_{j=0}^{\nu-1}(s+j)\Big)
 \sum_{k\in\N_0^p} P_{L,\nu}(k+a)\,f(k+a)^{-(s+\nu)},
$$
which is initially seen to be valid for $\re(s)\gg0$, actually gives
absolutely convergent expressions for $\re(s)>-N-\frac1m$. By
analytic continuation, the above holds at $s=-N$. However, at $s=-N$
the right-hand side vanishes rather trivially. Indeed, for $\nu>N$,
the product over $j$ vanishes at $s=-N$. For $0\le \nu\le N$, the
degree of $P_{L,\nu}$ is at most $q+m\nu-|L|\le q+mN-(Nm+q+p+1)<0$,
and so $P_{L,\nu}$ vanishes identically. We conclude that all
$a$-derivatives of $ \zeta(s;f_a,g_a)$ of order greater than
$Nm+q+p$ vanish. Hence $ \zeta(s;f_a,g_a)$ is a polynomial in $a$ of
degree at most $Nm+q+p$.

 Replacing  $\sum_{k\in\N_0^p}$
by $\int_{x\in\R_{\ge0}^p}$, the  above proof   shows that $
Z(s;f_a,g_a)$ is also a polynomial in $a$ of degree at most $Nm+q+p
\ \big($alternatively, this follows from Raabe's formula
\eqref{Raabe} and the polynomial nature of $a\to
\zeta(s;f_a,g_a)\,\big)$. This completes the proof of the second
claim in Proposition \ref{RaabeMainP}.

 The third claim  follows
from the following  lemma, with $Q(a):=\zeta(s;f_a,g_a)$ and
$P(a):=\zeta(s;f_a,g_a)$.
\begin{lemma}\label{LemmaB} Let $P$ and $Q$ be two polynomials in
$p$ variables linked by
\begin{equation} \label{RaabeT}
P(a)=\int_{t\in[0,1]^p} Q(a+t)\,dt.
\end{equation}
 Write out
\begin{equation} \label{coefficients2}
P(a_1,\ldots,a_p)=\sum_{L} d_L \prod_{i=1}^p a_i^{L_i}, \ \ \ \ \ \
\end{equation}
where $d_L\in\C$ and $L=(L_1,\ldots,L_p)\in\N_0^p$ ranges over a
finite set of multi-indices. Then
\begin{equation} \label{inverseRaabe}
 Q(a_1,\ldots,a_p)= \sum_{L} d_L
\prod_{i=1}^p B_{L_i}(a_i),
\end{equation}
where the $B_{L_i}(a_i)$ are Bernoulli polynomials, defined   in
\eqref{BernoulliDef}. Conversely, if $Q$ is given by
\eqref{inverseRaabe}, then   \eqref{RaabeT} and
\eqref{coefficients2} are equivalent formulas for $P$.
    \end{lemma}

\begin{proof} Let $V=V_{m,p}$ be the finite-dimensional
complex vector space of polynomials in $p$ variables
$a=(a_1\ldots,a_p)$, with complex coefficients  and having degree at
most $m$. Note that both $\{ a^L\}_L $ and $\{ B_L(a) \}_L $ are
$\C$-bases of $V$. Here $L=(L_1,\ldots,L_p)$ ranges over all
multi-indices with $|L|:=\sum_{i=1}^p L_i\le m$ and
$a^L:=\prod_{i=1}^p a_i^{L_i}$.  That $\{ B_L(a) \}_L $ is a basis
of $V=V_{m,p}$ can be proved by induction on $m$, since $a^L-B_L(a)$
has degree strictly less than $|L|$. Let $R:V\to V$ be the
$\C$-linear map taking $Q=Q(a)\in V$ to
$$
R(Q)(a):= \int_{t\in[0,1]^p} Q(a+t)\,dt.
$$
The lemma can be restated as saying that the inverse map to $R$
exists and takes $a^L$ to $B_L(a)$. Hence, it will suffice to show
that $R\big(B_L(a)\big)=a^L$, for then $R$ is an isomorphism (it
takes one basis to another). Using  \cite[p.\ 4]{3} \cite[pp.\
66--67]{9}
$$
\frac{d}{dx} B_{j+1}(x)=(j+1) B_j(x) \qquad\text{and}\qquad
B_j(x+1)-B_j(x)=jx^{j-1},
$$
we calculate
\begin{align}
R\big(B_L(a)\big)=& \int_{t\in[0,1]^p} B_L(a+t)\,dt=
\prod_{i=1}^p\int_0^1 B_{L_i}(a_i+t_i)\,dt_i\nonumber \\
=&   \prod_{i=1}^p\frac{1}{L_i+1}
\big(B_{L_i+1}(a_i+1)-B_{L_i+1}(a_i)\big)\nonumber =  \prod_{i=1}^p
a_i^{L_i}.
\end{align}
This concludes the proof of the lemma and of Proposition
\ref{RaabeMainP}.\end{proof}\end{proof}

 Using Proposition \ref{RaabeMainP}, we now  show that  Theorem \ref{MainTheorem}
follows from the product formula for  zeta integrals
 \begin{equation}\label{IntFormula2}
 \deg(f)\cdot Z\big(0;f,g\big)=\sum_{j=1}^n
\deg(f_j)\cdot Z(0;f_j,g),
\end{equation}
which we prove  in the next section. We always assume that all the
$f_j$ are polynomials in $p$ variables and satisfy Mahler's
Hypothesis. Theorem \ref{MainTheorem} states that
\begin{equation}\label{SeriesFormula2}
 \deg(f)\cdot\zeta\big(0;f,g\big)=\sum_{j=1}^n
\deg(f_j)\cdot\zeta(0;f_j,g)\qquad\qquad\Big( f:=\prod_{j=1}^n f_j
\Big),
\end{equation}
To prove this write
\begin{equation}\label{ZZZ}Z(0;f_a,g_a)=\sum_L c_L(0;f,g) a^L,\qquad
Z(0;\big(f_j\big)_a,g_a)=\sum_L c_L(0;f_j,g) a^L.
\end{equation}
Since $\deg(f_a)=\deg(f)$, replacing  $f$ by $f_a$, $f_j$ by
$\big(f_j\big)_a$ and $g$ by $g_a$ in   \eqref{IntFormula2}
 gives
$$
\deg(f) c_L(0;f,g)=\sum_{j=1}^n \deg(f_j)c_L(0;f_j,g)
$$
for all the coefficients $c_L(0;f,g)$ and $c_L(0;f_j,g) $ appearing
in \eqref{ZZZ}. Equation \eqref{Bern} now gives
\begin{align*}
 \deg(f)\zeta(0;f,g) &= \sum_L B_L \deg(f) c_L(0;f,g) \\
&=\sum_L B_L
\sum_{j=1}^n \deg(f_j)   c_L(0;f_j,g)  \\
&= \sum_{j=1}^n \deg(f_j)\sum_L B_L   c_L(0;f_j,g)=\sum_{j=1}^n
\deg(f_j)\zeta(0;f_j,g),
\end{align*}
 as claimed.

\section{Special values of zeta integrals}
\noindent In this section we first follow Mahler's proof \cite{11}
of the meromorphic continuation of  zeta integrals $Z(s;f,g)$. We
then show that Mahler's formulas simplify when $s$ is a non-positive
integer. Finally, we consider $s=0$ and show the product formula
\eqref{IntFormula2}.

Let us prove the part of Mahler's Theorem giving the meromorphic
continuation of $Z(s;f,g)$. Recall that in \eqref{Cubical} we
introduced cubical coordinates $x=\rho \sigma$ on $\R^p_{\ge0}$. Let
$d\sigma$ denote the natural $(p-1)$-dimensional volume element on
$\partial C_+^p$. A short Jacobian calculation shows that Lebesgue
measure $dx=\rho^{p-1} \,d\rho\,d\sigma\,$.\footnote{\ This formula
coincides formally with the Jacobian for spherical coordinates. This
means that  every formula below remains valid  on replacing
$\partial C_+^p$ by the spherical piece
$$S_+^{p-1}=\big\{(x_1,\dots,x_p)\in\R^p_{\ge0}
\big|\,x_1^2+x_2^2+\cdots+x_p^2=1 \big\}$$ and letting $d\sigma$
stand for the spherical $(p-1)$-dimensional volume form. However,
cubical coordinates  will usually result in simpler integrals, since
$\partial C_+^p$ is flat.}
   We have also seen that a branch of $\log f(x)$ can be chosen, continuous for $x\in\R^p_{\ge0}$
   and analytic locally in $f$. Note
 that the imaginary part $\im\big(\log f(x)\big)$ is uniformly
bounded  for $x\in\R^p_{\ge0}$.

 For $s$ in some compact set
$K\subset\C$ we have from \eqref{Branch} the
 estimate
$$
|g(x)f(x)^{-s}|\le c_K
\rho^{q-m\re(s)}\qquad\qquad(x=\rho\sigma\not=0,\
\,x\in\R^p_{\ge0}),
$$
for some constant $c_K$ independent of $x$. Hence the zeta integral
\eqref{Dirzeta} converges  when $q-m\re(s)<-p$, \ie in the right
half-plane $\re(s)>(q+p)/m$. Moreover, for such  $s$ and polynomials
$h$ in a neighborhood of $f$, the map $(s,h,g)\to Z(s;h,g)$ is
analytic in all three variables.

To prove Mahler's theorem for $Z(s;f,g)$ it suffices to show, for
each integer $N\ge0$, that $Z(s;f,g)$ extends meromorphically to the
right half-plane $\re(s)>-N-\frac1m$, is regular at $s=-N$, and that
all of its poles in the half-plane $\re(s)>-N-\frac1m$ are at most
simple and occur among $s$ of the form $s=(q+p-\ell)/m$, with
$\ell\ge0$ an integer. First take $\re(s)>(q+p)/m$, choose any $
w>0$ (taken sufficiently large  below) and write using cubical
coordinates \eqref{Cubical},
\begin{align}
 & Z(s;f,g)=\int_{\rho=0}^\infty  \int_{\sigma\in \partial C_+^p}
\rho^{p-1}g(\rho\sigma)f(\rho\sigma)^{-s}\,d\sigma d\rho\nonumber\\
&=\int_{\rho=0}^w \int_{\sigma\in \partial C_+^p}
\rho^{p-1}g(\rho\sigma)f(\rho\sigma)^{-s}\,d\sigma d\rho\nonumber\\
\label{Z1Z2} & {\phantom{X}}+
 \int_{\rho=w}^\infty\int_{\sigma\in
\partial C_+^p} \rho^{p-1}g(\rho\sigma)f(\rho\sigma)^{-s}\,d\sigma
d\rho  =:Z_1(s,w)+Z_2(s,w).
\end{align}
The integral $Z_1(s,w)$ over the compact set $[0,w]\times \partial
C_+^p$ gives an entire function of $s$, so we turn to the
continuation of $Z_2(s,w)$.  Using the factorization \eqref{Factorf}
we find
\begin{equation}\nonumber
    Z_2(s, w) =  \IO\int_w^{\infty} \rho^{p-1-ms} f_\mathrm{top}(\sigma)^{-s}
    g(\rho\sigma) \big(1+r(\rho\sigma)\big)^{-s}\,d\rho\,d\sigma  .
 \end{equation}
 Following
Mahler we replace $\big(1+r(\rho\sigma)\big)^{-s}$ by its finite
Taylor expansion. For $k$-times continuously differentiable
$G:[0,1]\to\C $
  we have for $ k\ge1$ and $0\le y\le1 $ \cite[\S5$\cdot$41]{17}
  $$ G(y) = \sum_{\lambda=0}^{k-1} \!\frac{G^{(\lambda)}(0)}{\lambda!}
      y^\lambda + \frac{1}{(k-1)!} \int_0^y\!G^{(k)}(t)
     (y-t)^{k-1} \,dt. $$
  For  $0\le t\le 1$ and $r\in\C$ with $|r|<1$, let
$$G(t)=G_r(t) := \big(1+t r\big)^{-s},$$ where we use the principal
branch. Then
   $$ \frac{G^{(\lambda)}(0)}{\lambda!} =
   r^\lambda\frac{\prod_{j=0}^{\lambda-1}(-s-j)}{\lambda!}
= r^\lambda\binom{-s}{\lambda} . $$
  The remainder term for $y=1$ is
   $$ \frac{1}{(k-1)!} \int_0^1 G^{(k)}(t)
      (1-t)^{k-1} \,dt = k r^k \binom{-s}{k} \int_0^1
      \frac{(1-t)^{k-1}}{ (1+t r )^{s+k}}\,dt. $$
  Hence, $G(1)= (1+r  )^{-s} $ is given by
  \begin{eqnarray*}
     (1+r )^{-s} = \sum_{\lambda=0}^{k-1}
    \binom{-s}{\lambda} r ^{\lambda} + k \binom{-s}{k}
    r ^k \int_0^1 \frac{(1-t)^{k-1}}{ (1+t\,r )^{s+k}}\,dt.
  \end{eqnarray*}
For $w\ge\rho_f$   (large enough that $|r(\rho\sigma)|\le\frac12$
for $\rho\ge w$), $\re(s)>(q+p)/m$ and $N\ge0$ an integer,
$Z_2(s,w)$
 $\big($see \eqref{Z1Z2}$\big)$ can now be written
  \begin{equation}\label{Z20}
    Z_2(s, w) = k \binom{-s}{k} N_k(s, w)+\sum_{\lambda=0}^{Nm+q+p} \binom{-s}{\lambda} M_{\lambda}(s,w),
  \end{equation}
  where $k:=Nm+q+p+1$, and $N_k=N_k(s, w)$ and $M_\lambda$ are given
  by
    \begin{equation}
    N_k :=    \IO\!\int_w^{\infty}\!  f_\mathrm{top}(\sigma)^{-s} \rho^{p-1-ms} g(\rho\sigma)
     r(\rho\sigma)^k\!\!
    \int_0^1\!\frac{(1-t)^{k-1}}{\big(1+t\,r(\rho\sigma)\big)^{s+k}}
     dt d\rho d\sigma    \label{enek}
  \end{equation}
  \begin{equation}
    M_{\lambda}(s, w) :=   \IO\int_w^{\infty}  f_\mathrm{top}(\sigma)^{-s} \rho^{p-1-ms}  g(\rho\sigma)
    r(\rho\sigma)^{\lambda}\,d\rho\,d\sigma. \label{emel}
    \end{equation}

We now extend $N_k$ analytically in $s$. Since $|r(\rho\sigma)|$
decreases at least like $\rho^{-1}$ as $\rho\to\infty$, the
integrand   in \eqref{enek} decreases at least like
$$\rho^{p+q-k-m\re(s)-1}=\rho^{-m(N+\frac1m+\re(s))-1}.$$
 Hence
$N_k(s,w)$ extends to an analytic function  in the right half-plane
$\re(s)>-N-\frac1m$.

To get the meromorphic continuation of $M_\lambda$  expand
\begin{equation} \label{Expand}
    g(\rho\sigma)\,r(\rho\sigma)^{\lambda} = \rho^{q-\lambda} \sum_{h = 0}^{q+(m-1)\lambda}
    A_{\lambda,h}(\sigma) \,\rho^{-h}\quad\big(q:=\deg(g),\
    \,m:=\deg(f)\big),
  \end{equation}
  where, in view of \eqref{Writer} and Mahler's Hypothesis on $f$, the $A_{\lambda,h}$ are rational
  functions with no poles in a neighborhood of $\partial C_+^p$.
From \eqref{emel} and \eqref{Expand} we find,  for $\re(s)>(q+p)/m$,
 \begin{align}
      M_{\lambda}&(s, w) = \sum_{h=0}^{q+(m-1)\lambda} \IO\!\int_{\rho=w}^{\infty}
        \rho^{q+p-ms-\lambda-h-1} A_{\lambda,h}(\sigma)\,f_\mathrm{top}(\sigma)^{-s} \,d\rho
        \,d\sigma \nonumber \\
     &= \label{mero}\sum_{h=0}^{q+(m-1)\lambda} \frac{w^{q+p-ms-\lambda-h}}{ms+\lambda+h-q-p}
        \IO\!A_{\lambda,h}(\sigma)\,f_\mathrm{top}(\sigma)^{-s}\,d\sigma.
    \end{align}
The above expression gives a meromorphic continuation of
$M_\lambda(s,w)$ to all $s\in\C$, with at most  simple poles at
rational points of the form $s=\frac{q+p-\ell}{m}$, with
$\ell=\lambda+h$ an integer in the range $\lambda\le\ell\le
q+m\lambda$. As $\lambda\ge0$, \eqref{Z20} shows that it only
remains to prove that $Z_2(s,w)$ is regular at $s=-N$.

For $\lambda>N$ there is no pole of $Z_2(s,w)$ at $s=-N$ because of
the factor $\binom{-s}{\lambda}$ multiplying $M_\lambda$ in
\eqref{Z20}. Finally, $M_\lambda$ has no pole at $s=-N$ for
$0\le\lambda\le N$ since its left-most pole  occurs at
$s=\frac{q+p-(q+m\lambda)}{m}>-\lambda\ge -N$.

The above proof  $\big($mainly equations \eqref{Z1Z2} to
\eqref{mero}$\big)$ shows that the analytic continuation obtained
for $Z(s;h,g)$ depends analytically on the coefficients of the
polynomials involved. This concludes the proof of Mahler's Theorem
for $Z(s;f,g)$.

\newpage

 We now show that Mahler's
formulas above for $Z(s;f,g) $ simplify  at non-positive integers
$s=-N$.
\begin{theorem}\label{Values} Let $f\in\C[x]$ be a polynomial
 in $p$ variables satisfying Mahler's Hypothesis, let
 $g\in\C[x]$   be any polynomial in $p$ variables, and let $N\ge 0$ be
 a non-negative integer. Then
 the value of the analytic continuation of the
zeta integral \eqref{Dirzeta} at $s=-N$ is
 \begin{equation} \label{valorZ}
    Z(-N;f,g) =\frac1m\,  \sum_{\lambda = N+\lceil p/m \rceil}^{q+p+Nm}
      \frac{(-1)^{\lambda-N}}{\lambda - N} \frac{1}{\binom{\lambda}{N}}
      \int_{\sigma\in \partial C_+^p}\!C_{\lambda,N}(\sigma) \, f_\mathrm{top}(\sigma)^N \,d\sigma,
  \end{equation}
where $m=\deg(f),\ \, q=\deg(g), \ \,\lceil p/m \rceil$ is the
smallest integer $\ge p/m$, the integral is over $\partial C_+^p$
defined in \eqref{Hypercube}, $C_{\lambda,N}(\sigma)$ is the
coefficient of $\rho^{-p-mN}$ in the rational function
$g(\rho\sigma)r_f(\rho\sigma)^\lambda$, with $r_f$ as in
\eqref{Writer}, and $f_\mathrm{top}$ is the degree-$m$ part of $f$.
\end{theorem}
\begin{proof} Combining \eqref{Z1Z2} and \eqref{Z20} in the proof of
Mahler's theorem, we find for $\re(s)>-N-\frac1m,\
\,\,k:=q+p+Nm+1>N$ and any large enough $w$,
\begin{equation}\nonumber
    Z(s;f,g) = Z_1(s,w)+k \binom{-s}{k} N_k(s, w)+
    \sum_{\lambda=0}^{Nm+q+p} \binom{-s}{\lambda} M_{\lambda}(s,w).
  \end{equation}
Since $N_k$ and $M_\lambda$ are analytic at $s=-N$ for $\lambda\le
N$, we find
\begin{align}\label{ZNw}
    Z(-N;f,g) = Z_1(-N,w)+&
    \sum_{\lambda=0}^{N} \binom{N}{\lambda} M_{\lambda}(-N,w)\\ &
    +\lim_{s\to-N}\sum_{\lambda=N+1}^{Nm+q+p} \binom{-s}{\lambda} M_{\lambda}(s,w).
 \nonumber \end{align}
We shall now see that $Z_1(-N,w)$ and the first sum above cancel.
Using \eqref{Factorf}, the binomial expansion and \eqref{Expand} we
have
\begin{align}
Z_1(-N,w) & = \int_{ \partial C_+^p}\int_{0}^w \rho^{p-1}
g(\rho\sigma) f(\rho\sigma)^N\, d\rho d\sigma\nonumber
 \\  &=\int_{\partial C_+^p}f_\mathrm{top}(\sigma)^N\int_{0}^w
\rho^{mN+p-1}g(\rho\sigma) \big(1+r(\rho\sigma)\big)^N \, d\rho
d\sigma\nonumber \\ &= \sum_{\lambda=0}^N \binom{N}{\lambda} \int_{
\partial C_+^p}f_\mathrm{top}(\sigma)^N\int_{0}^w
\rho^{mN+p-1}g(\rho\sigma) r(\rho\sigma)^\lambda\, d\rho
d\sigma\nonumber
 \\ &\hskip-1.3cm= \sum_{\lambda=0}^N\binom{N}{\lambda}\sum_{h=0}^{q+\lambda(m-1)}
\int_{ \partial C_+^p}f_\mathrm{top}(\sigma)^N A_{\lambda,h}
(\sigma) \int_{0}^w \rho^{mN+p-1+q-\lambda-h}\, d\rho
d\sigma\nonumber
\\  &\hskip-1.3cm=\sum_{\lambda=0}^N\binom{N}{\lambda}\sum_{h=0}^{q+\lambda(m-1)}
\frac{w^{mN+p+q-\lambda-h} }{mN+p+q-\lambda-h}\int_{ \partial C_+^p}
f_\mathrm{top}(\sigma)^N A_{\lambda,h} (\sigma) \, d\sigma.
 \nonumber
\end{align}
By \eqref{mero}, however, the above is just $ - \sum_{\lambda=0}^{N}
\binom{N}{\lambda} M_{\lambda}(-N,w)$. Note that there is no
singularity of the integrals  above at $\rho=0$ since
$$
mN+p-1+q-\lambda-h\ge
mN+p-1+q-\lambda-\big(q+\lambda(m-1)\big)=m(N-\lambda)+p-1,
$$
which is clearly non-negative for $\lambda\le N$.
 Returning to \eqref{ZNw} we now
have
 \begin{equation} \label{valorZ1}
Z(-N;f,g)=\sum_{\lambda=N+1}^{p+q+Nm}\lim_{s\to-N}\binom{-s}{\lambda}M_\lambda(s,w).
\end{equation}
We have seen from \eqref{mero} that $M_\lambda$ may have simple
poles at points $s=\frac{q+p-\ell}{m}$ with $\ell=\lambda+h$ in the
range $\lambda\le\ell\le q+m\lambda$. Setting $s=-N$ we find
$\ell=q+p+Nm$, so $h=q+p+Nm-\lambda$. As $0\le h\le q+(m-1)\lambda$
in \eqref{mero}, $M_\lambda$ can have a pole at $s=-N$ only if
$$
h=q+p+Nm-\lambda \le q+(m-1)\lambda , \qquad\qquad \text{\ie}\ \
\lambda\ge N+\frac{p}{m}.
$$ Since $\lambda$ is an integer,  $\lambda\ge  N+\lceil p/m
\rceil$, whence from \eqref{valorZ1} we have
\begin{equation} \nonumber
Z(-N;f,g)=\sum_{\lambda=N+\lceil p/m
\rceil}^{p+q+Nm}\Big(\lim_{s\to-N}\frac{\binom{-s}{\lambda}}{s+N}\Big)\cdot
\mathrm{Residue}_{s=-N}\big(M_\lambda(s,w)\big).
\end{equation}
Induction on $\lambda$ (with $\lambda\ge N+1$) shows
$$
\lim_{s\to-N}\frac{\binom{-s}{\lambda}}{s+N}=
\frac{(-1)^{\lambda-N}}{\lambda - N} \frac{1}{\binom{\lambda}{N}}.
$$
From \eqref{mero}, on the other hand, we find
$$
\mathrm{Residue}_{s=-N}\big(M_\lambda(s,w)\big)=\frac1m
        \IO\!A_{\lambda,q+p+Nm-\lambda}(\sigma)\,f_\mathrm{top}(\sigma)^{N}\,d\sigma
        .
$$
 Combining the last three equations we see
that we have proved Theorem \ref{Values}, but with $C_{\lambda,N}$
replaced by $A_{\lambda,q+p+Nm-\lambda}$. Examining the definition
of the $A_{\lambda,h}$ in \eqref{Expand}, we see that
$A_{\lambda,h}(\sigma)$ is the coefficient of $\rho^{q-\lambda-h}$
in $g(\rho\sigma)r(\rho\sigma)^\lambda$. Thus
$A_{\lambda,q+p+Nm-\lambda}$ is the coefficient of
 $\rho^{q-\lambda-(q+p+Nm-\lambda)}=\rho^{-p-Nm}$, \ie $C_{\lambda,N}$
 in \eqref{valorZ}.\end{proof}

To apply Theorem \ref{Values} to $N=0$,  recall from \eqref{Writer}
that for fixed $\sigma\in \partial C_+^p$, $ \frac{f\ \
}{f_\mathrm{top}}(\rho\sigma) =1+r(\rho\sigma)$ is an analytic
function of $1/\rho$ (considering $\rho$ now as complex with
$|\rho|\gg0$) and that $r(\rho\sigma)\to0$ as $|\rho|\to\infty$.
Thus the principal value $\log\!\big( 1+r(\rho\sigma)\big)$ is an
analytic function of $1/\rho$ for $1/\rho$ in a disc near 0. In
particular, the coefficient of any power of $1/\rho$ in the Laurent
expansion of $\log\!\big( 1+r(\rho\sigma)\big)$  is well-defined.

\begin{Coro}\label{Coro1} Let $f\in\C[x]$ be a polynomial
 in $p$ variables satisfying Mahler's Hypothesis
and let $g$ be any polynomial in $p$ variables. Then
 $\deg(f)\cdot Z(0;f,g)$ is
 the coefficient of $\rho^{-p}$ in the Laurent expansion   of
\begin{equation}\nonumber
-\int_{\sigma\in \partial C_+^p} g(\rho\sigma) \log\!\Big( \frac{f\
\ }{f_\mathrm{top} }(\rho\sigma)\Big)\,d\sigma.
\end{equation}
\end{Coro}

Before we prove Corollary \ref{Coro1},  we use it to prove the
product formula \eqref{IntFormula2}, \ie
$$
\deg\!\bigg(\! \prod_{j=1}^n f_j\bigg)\cdot Z\big(0;\prod_{j=1}^n
f_j,g\big)=\sum_{j=1}^n\deg(f_j)\cdot Z(0;f_j,g),
$$
where we assume that all the polynomials $f_j$ satisfy Mahler's
Hypothesis (with the same $p$). Induction on $n$ shows that it
suffices to deal with the case $n=2$. Suppose then that $f_1$ and
$f_2$  satisfy Mahler's Hypothesis. Since $(f_1\cdot
f_2)_\mathrm{top}= \big(f_1\big)_\mathrm{top} \cdot
\big(f_2\big)_\mathrm{top}$, the product $f_1f_2$ also satisfies
Mahler's Hypothesis. From
$$
\frac{f_1f_2}{\ \ \, \,\big(f_1f_2\big)_\mathrm{top}}=\frac{f_1}{\
\,\ \big(f_1\big)_\mathrm{top}}\cdot\frac{f_2}{\ \ \
\big(f_2\big)_\mathrm{top}}\,,
$$
and Corollary \ref{Coro1} we have (letting
$\mathrm{Coeff}_{\rho^{-p}}$ stand for the coefficient of
$\rho^{-p}$)
\begin{align}
\deg&(f_1 f_2)Z(0;f_1 f_2,g)=
-\mathrm{Coeff}_{\rho^{-p}}  \int_{\sigma\in \partial
C_+^p}\! g(\rho\sigma)
\log\!\Big(\!\frac{f_1 f_2}{ \, \,\big(f_1 f_2\big)_\mathrm{top}}(\rho\sigma)\Big)\,d\sigma \nonumber \\
=& -\mathrm{Coeff}_{\rho^{-p}}  \int_{\sigma\in \partial
C_+^p}\! g(\rho\sigma) \bigg(\!\!\log\!\Big(\!\! \frac{f_1}{ \,
\big(f_1\big)_\mathrm{top}}(\rho\sigma)\!\Big)+\log\!\Big(\!
\frac{f_2}{
 \big(f_2\big)_\mathrm{top}}
(\rho\sigma)\Big)\bigg)\,d\sigma 
\nonumber \\
=& \deg(f_1)Z(0;f_1,g)+\deg(f_2)Z(0;f_2,g).\nonumber
 \end{align}
Note that it is legitimate to use
$$
\log\!\Big(\frac{f_1 \ }{\big(f_1\big)_\mathrm{top}}\cdot\frac{f_2
\, }{\big(f_2\big)_\mathrm{top}}\Big)=\log\!\Big(\frac{f_1 \
}{\big(f_1\big)_\mathrm{top}}\Big)+\log\!\Big(\frac{f_2 \
}{\big(f_2\big)_\mathrm{top}}\Big)
$$  in the second equation above since we only
take (principal value) logarithms of complex numbers near 1 as
$\rho\to\infty$.

\vskip.3cm \noindent 
{\bf{Proof}} (of Corollary \ref{Coro1}). Applying Theorem
\ref{Values} at $s=0$ we find
 \begin{equation} \label{Valor0}
   \deg(f) Z(0;f,g) = \int_{\sigma\in \partial C_+^p} \mathrm{Coeff}_{\rho^{-p}}
   \bigg(   g(\rho\sigma)\sum_{\lambda = \lceil p/m \rceil}^{q+p}
      \frac{(-1)^{\lambda}}{\lambda} r(\rho\sigma)^\lambda
        \bigg)\,d\sigma.
  \end{equation}
  The above sum would be a logarithm  if only we could extend
  the sum to all  $\lambda\ge1$. This is possible since
 \begin{equation} \label{Coeff0}
\mathrm{Coeff}_{\rho^{-p}}\left(g(\rho\sigma)r(\rho\sigma)^\lambda\right)=0
\end{equation}
  for $1\le\lambda<p/m$ and also for $\lambda>q+p$. Indeed,
   the powers   $\rho^t$ appearing in
 the $\rho$-expansion of $g(\rho\sigma)r(\rho\sigma)^\lambda$
 in \eqref{Expand} are all in the range
 $-m\lambda\le t\le q-\lambda$. For $t=-p$, this amounts to
  $\frac{p}{m}\le \lambda\le p+ q$. Thus,
 outside this range there is no coefficient of $\rho^{-p}$ in
 $g(\rho\sigma)r(\rho\sigma)^\lambda$.

   Since $
  |r(\rho\sigma)|<\frac12$ for large enough $\rho$, we
  obtain a convergent series on letting  $\lambda$
   in \eqref{Valor0}
   range over all $\N$. From \eqref{Valor0}, \eqref{Coeff0} and \eqref{Factorf} we get
  \begin{align} \nonumber
   \deg(f) Z(0;f,g)& = \int_{\sigma\in \partial C_+^p}
    \mathrm{Coeff}_{\rho^{-p}}\bigg( g(\rho\sigma)  \sum_{\lambda =
   1}^\infty
      \frac{(-1)^{\lambda}}{\lambda} r(\rho\sigma)^\lambda
        \bigg)\,d\sigma\nonumber\\ &=
         -\mathrm{Coeff}_{\rho^{-p}}\bigg(\int_{\sigma\in
        \partial C_+^p} g(\rho\sigma)
        \log
        \!\big(1+r(\rho\sigma)\big)
       \,d\sigma \bigg)\nonumber\\ &=
       -\mathrm{Coeff}_{\rho^{-p}}\bigg(\int_{\sigma\in \partial C_+^p}
        g(\rho\sigma)\log
        \!\Big(\frac{f}{f_\mathrm{top}} (\rho\sigma)\Big)
        \,d\sigma\bigg), \nonumber
  \end{align}
concluding the proof.

\end{document}